\documentclass[11pt]{amsart}

\usepackage[mathscr]{eucal}
\usepackage{amsmath,amssymb,amscd,amsthm}
\usepackage{color} 
\usepackage{epsfig}
\newtheorem{Theorem}{Theorem}
\newtheorem{Lemma}{Lemma} 

\newtheorem{Definition}{Definition}  
\newtheorem{Corollary}{Corollary}

\theoremstyle{definition} 
\newtheorem{Remark}{Remark} 

\theoremstyle{plane}

\def \beq{ \begin{equation} }
\def \eeq{\end{equation}}

\title{On certain  M\"{o}bius type solutions for the $n$--body problem in a positive space form}
   
\begin{document}

\maketitle

\markboth{Ortega-Palencia Pedro Pablo and  Reyes-Victoria J. Guadalupe}{ Solutions of type M\"{o}bius for the $n$--body problem in a positive space form}

\vspace{-0.5cm}

\author{
\begin{center}
{\rm PEDRO PABLO ORTEGA PALENCIA \\
         Departamento de Matem\'aticas  \\
         Universidad de Cartagena \\
         Cartagena de Indias, \\
         COLOMBIA\\
         {\tt portegap@unicartagena.edu.co}\\
        \medskip
         J. GUADALUPE REYES VICTORIA \\
         Departamento de Matem\'aticas \\
         UAM-Iztapalapa \\
         M\'exico, D.F. \\
         MEXICO \\
         {\tt revg@xanum.uam.mx}}
\end{center}}

 \bigskip
\begin{center}
\today
\end{center}

\begin{abstract}
We study here the  M\"{o}bius type  solutions for the $n$-body problem in  a two dimensional positive space form $\mathbb{M}_R^2$. With methods of M\"{o}bius
geometry and using the Iwasawa decomposition of the M\"{o}bius group of automorphisms ${\rm \bf  Mob}_2 \, (\mathbb{M}_R^2)$, we state algebraic functional  conditions for the existence of such  type  of solutions in  $\mathbb{M}_R^2$. We mention some examples of these type of solutions.
\end{abstract}

\medskip

\footnote*{MSC: Primary 70F15, Secondary 53Z05}

\smallskip

\footnote*{Keywords: Two dimensional  positive space form, M\"{o}bius geometry,  The $n$--body problem.}

\section{Introduction}
\label{sec:intro}

We consider in this paper the problem of studying the motion of $n$ point interacting particles of masses $m_1,\cdots,m_n$ moving on a
positive two dimensional space form $\mathbb{M}^2_R$ under the action of a suitable (cotangent) potential.

This work  is a little branch of the forgotten, unpublished, incomplete, but full of ideas work \cite{Reyes-Perez} with E. P\'erez-Chavela, to whom the second author thanks the fact of introducing him in the study of the curved celestial mechanics without leaving his profile of geometer.
We omit along all the document the term {\it curved} because, as in \cite{Reyes}, the studied space has a non-euclidian metric. Also, we omit the term 
{\it intrinsic}, because when we have chosen the geometric structure on $\mathbb{M}_R^2$ as   Riemann surface $\widehat{\mathbb{C}}$ with the canonical complex variables $(z,\bar{z})$ endowed with the conformal metric
\begin{equation}\label{met-c}
   ds^2= \frac{4R^4 \,dz d\bar{z}}{(R^2 + |z|^2)^2},
\end{equation}
 we have given the corresponding  differentiable structure in such coordinates to this space  (see \cite{Farkas} for more details).

 Following the methods of the geometric Erlangen program
as in \cite{Kisil,Reyes-Perez,Reyes} for the space $\mathbb{M}_R^2$, we define the conic motions of the $n$--body problem in terms of the
action of one dimensional subgroups of M\"{o}bius transformations group ${\rm \bf  Mob}_2 \, (\mathbb{M}_R^2)$ associated to a suitable  vector fields. By the method of matching
the vector field in the Lie algebra associated to the corresponding
subgroup, with the cotangent gravitational field, as we did in  \cite{Diacu2}
\cite{Perez}, \cite{Reyes-Perez}  and \cite{Reyes}, we state in each case functional algebraic
conditions (depending on the time $t$) which the solutions must hold in order to be one of such M\"{o}bius solutions.

\smallskip

We organize  the  paper as follows.

\smallskip

In section  we state the Iwasawa decomposition of ${\rm \bf  Mob}_2 \, (\mathbb{M}_R^2)$ that allows us to obtain
 the algebraic and geometric classification of all M\"{o}bius transformations in {\it elliptics, hyperbolics}, and {\it parabolics}. The action of such one parametric transformations  will generate the corresponding conic curves of the M\"{o}bius geometry  of $\mathbb{M}^2_R$.

\smallskip

In section \ref{sec:equamotion} we state the equations of motion of
the problem in  complex coordinates as in \cite{Diacu2}, \cite{Perez}, \cite{Reyes-Perez}, but they are obtained as in \cite{Reyes} . 

\smallskip

In section \ref{elliptic} we define the  M\"{o}bius elliptic solutions, thorough the action of suitable one-dimensional parametric subgroups of isometries
in $SU(2)$ associated to suitable Killing vector fields in the Lie algebra $su(2)$. There, by matching the Killing vector fields with the gravitational one, we find the main algebraic conditions in order of having  a  M\"{o}bius  elliptic solution. We give new examples of such type of solutions for this problem for which the geodesic circle and the southern and northern tropics play an important role.

\smallskip

In section \ref{hyperbolic} we define the so called  M\"{o}bius hyperbolic
solutions, via the action of the one-dimensional parametric subgroup
of M\"{o}bius hyperbolic transformations
 showing that they correspond to some particular type of the
homothetic orbits found in \cite{Diac, Diac-Perez}. We also give the
functional algebraic conditions (depending on the time $t$) on the positions of the particles for having
one of such solutions.

\smallskip

In section \ref{parabolic} we define the  M\"{o}bius nilpotent parabolic solutions. As in
the previous cases we give the necessary and sufficient functional algebraic
conditions which such orbits must hold. We show that, as in \cite{Diacu2, Reyes},
such type of motions do no exist.

\section{The group ${\rm \bf  Mob}_2 \, (\widehat{\mathbb{C}})$ and the invariants of the  M\"{o}bius geometry}\label{sec:mobius-group}

We give an algebraic classification of the
 M\"{o}bius transformations defined on the extended complex plane  $\widehat{\mathbb{C}}= \mathbb{M}_R^2 \cup \{\infty\}$,
which corresponds to  the  Riemann sphere of radius $R$ endowed with the metric (\ref{met-c}). The reader interested in the details on all the objects 
shown in this section can see them in the aforementioned reference  \cite{Reyes-Perez} or in  \cite{Dub, Iwa,  Kisil, Kob}.

\begin{Definition}
A M\"{o}bius transformation is a fractional linear transformation $ f_A
: \widehat{\mathbb{C}}  \to \widehat{\mathbb{C}}$,
\[f_A (z) = \frac{a z+b}{c z + d}, \]
where $a,b,c,d \in \mathbb{C}$ and $ad-bc =1$, and the set of these automorphisms is the  group denoted by
 ${\rm \bf  Mob}_2 \, (\widehat{\mathbb{C}})$ and named the  M\"{o}bius group.
\end{Definition}

 Any M\"{o}bius
transformation $f_A$ is associated to some matrix
\[
   A= \left(\begin{array}{cc}
    a   &  b     \\
   c &  d   \\
    \end{array}\right) \in {\rm SL}(2, \mathbb{C})= \{A \in {\rm GL}(2, \mathbb{C}) \, | \, {\rm det} \, A =1 \},
\]
which defines an isomorphism between the groups $\displaystyle {\rm \bf  Mob}_2 \, (\widehat{\mathbb{C}}) $ and $\displaystyle {\rm SL} (2,\mathbb{C}) \, / \{ \pm I\}$.

 The special unitary subgroup  is 
\[{\rm SU}(2) = \{ A \in {\rm SL}(2,\mathbb{C}) \, | \, \, \bar{A}^T \,A= I \},  \]
and  each matrix
$A \in {\rm SU}(2)$ has the form
\[
   A= \left(\begin{array}{cc}
    a        &  b     \\
   -\bar{b} & \bar{a}   \\
    \end{array}\right),
\]
with  $a,b \in \mathbb{C}$ satisfying  $|a|^2 +|b|^2 =1$ (see \cite{Dub}).

We also know from   the Lie group theory that ${\rm SU}(2)$ is the maximal compact subgroup of $SL(2, \mathbb{C})$, and that the {\it group of proper isometries}  of $\mathbb{M}^2_R$  is the quotient $ \displaystyle {\rm SU}(2) \, /  \, \{ \pm I \}$.

\subsection{The elliptic  M\"{o}bius group }

The {\it Lie algebra} of ${\rm SU}(2)$ is the 3-dimensional real linear space
\[ {\rm su}(2) = \{ X \in {\rm M}(2,\mathbb{C}) \, | \, \, \bar{X}^T =-X, \, \, {\rm trace} \,(X)=0  \}  \]
spanned by the basis of complex Pauli's spinor matrices,
\[ \left\{
X_1 =  \left( \begin{array}{ccc}
    0 & 1 \\
    -1 & 0  \\
    \end{array}\right),  \quad
X_2 = \left(\begin{array}{cc}
    i & 0 \\
    0 & -i  \\
    \end{array}\right), \quad
X_3 =\left(\begin{array}{cc}
    0 & i \\
    i & 0  \\
    \end{array}\right)
\right\}, \]
which are  Killing vector fields on the Lie
group. Applying the exponential application to the lines $tX_1, tX_2$ and $tX_3$ into ${\rm SU}(2)$
we obtain the respective isometric one-dimensional subgroups, 
\begin{enumerate}
\item[\bf 1.]  The subgroup
\[ \exp (t \, X_1)=
\left(\begin{array}{cc}
    \cos t & \sin t \\
    -\sin t  & \cos t  \\
    \end{array}\right),
\]
which defines the one-parametric family of acting M\"{o}bius
transformations in $\widehat{\mathbb{C}}$,
\begin{equation} \label{eq:first-vector-field}
 f_1 (t, z) = \frac{ z \cos t  +\sin t}{ - z \sin t  + \cos t}.
\end{equation}

This flow has asociated the vector field $1+z^2$
 in $\widehat{\mathbb{C}}$ obtained
when we derive (\ref{eq:first-vector-field}) respect to the
parameter $t$ and evaluate it at  $t=0$. This vector field corresponds to the complex differential equation
\begin{equation} \label{eq:first-vector-field-1}
\dot{z}= 1+z^2 ,
\end{equation}
whose solutions  correspond to coaxal circles with fixed points $z_1=i$ and $z_2=-i$.

 Such that flow defines a foliation of $\mathbb{M}^2_{R}$, which divides it in two connected components (hemispheres) with a common border in the geodesic separatrix ${\rm Im} \, z=0$ (meridian circle). Each component
is foliated by the circular periodic orbits, solutions of the differential equation (\ref{eq:first-vector-field-1})
and the fixed points are the centers of such foliations, also called {\it focus} of the whole set of  circles  in the sense of the  M\"{o}bius geometry of $\mathbb{M}^2_{R}$, as in \cite{Kisil}. They can be seen in the sphere $\mathbb{S}^2_{R}$, up a suitable rotation,  as two sets of periodic concentric solutions, one sited inside of the north hemisphere bounded by the geodesic separatrix (equator), and other one sited inside of the south hemisphere bounded also by such geodesic (see \cite{Reyes-Perez} for details). 

\item[\bf 2.]  The subgroup
\[ \exp (t \, X_2)=
\left(\begin{array}{cc}
    e^{it} & 0 \\
    0  & e^{-it} \\
    \end{array}\right),
\]
which defines the one-parametric family of acting M\"{o}bius
transformations
\begin{equation} \label{eq:second-vector-field}
 f_2  (t,z) = e^{2it} \, z.
\end{equation}

This flow is associated to the vector field $iz$ in $\widehat{\mathbb{C}}$ and corresponds to the complex 
differential equation,
\begin{equation} \label{eq:second-vector-field-1}
\dot{z}= 2 i \, z.
\end{equation}

The  orbits of the action of the one-parametric subgroup
$\{ \exp(tX_2)\}$ in $\mathbb{M}_R^2$ are circular orbits with center in fixed point $z=0$.

\item[\bf 3.]  The subgroup
\[ \exp (t \, X_3)=
\left(\begin{array}{cc}
    \cos t & i\sin t \\
    i\sin t  & \cos t  \\
    \end{array}\right),
\]
which defines the one-parametric family of acting M\"{o}bius
transformations
\begin{equation} \label{eq:third-vector-field}
 f_3 (t,z) = \frac{ z \cos t  + i \sin t}{z \,
i\sin t  + \cos t}.
\end{equation}

This flow is associated to the vector field $i(1-z^2)$ in $\widehat{\mathbb{C}}$ which corresponds to the complex differential equation
\begin{equation} \label{eq:third-vector-field-1}
\dot{z}= i(1-z^2),
\end{equation}
whose orbits correspond also to coaxial circles  with  focus in the fixed points $z=-1$ and $z=1$.

 Also, as in the first case, this flow defines a foliation of $\mathbb{M}^2_{R}$, which divides it in two connected components (hemispheres) with a common border in the geodesic separatrix ${\rm Re} \, z=0$ (see \cite{Reyes-Perez} for details).

\end{enumerate}

We will use the above one-dimensional subgroups for obtaining the whole set of the so
called {\it  M\"{o}bius elliptic solutions} (or, as in the dynamical literature: {\it Relative equilibria solutions}) of the $n$-body problem in $\mathbb{M}_R^2$ .

\subsection{The hyperbolic normal  M\"{o}bius group}

In the Lie algebra  $sl(2,\mathbb{C})$ we have the hyperbolic vector field,
\[
X_4 = \frac{1}{2} \left( \begin{array}{ccc}
    1 & 0 \\
     0 & -1  \\
    \end{array}\right). \]
    
    If we consider also the straight line $\{t X_4\}$ in $sl(2,\mathbb{C})$,  it is applied under the exponential map
into the  one-parametric subgroup of transformations associated to the normal  hyperbolic one dimensional 
subgroup of  M\"{o}bius matrices of ${\rm \bf  Mob}_2 \, (\widehat{\mathbb{C}})$ of the form
\[ G_h (t)=
\left(\begin{array}{cc}
    e^{t/2} & 0 \\
    0  & e^{-t/2} \\
    \end{array}\right).
\]
The above matrices generate the  acting one dimensional parametric subgroup
of   M\"{o}bius transformations given by
\begin{equation}\label{eq:Hyperbolic}
f_{G_h(t)} (z) = e^{t} \, z.
\end{equation}

The flow of this group is associated with vector field $z$ in $\widehat{\mathbb{C}}$, and with the  first order complex differential equation
\begin{equation}\label{eq:hyperbolic-vector-field}
\dot{z} = z
\end{equation}

Such that flow is a set of straight lines arising in the origin of
 coordinates with fixed point  $z=0$.

We will use this one-dimensional subgroup for obtaining the set of the {\it  M\"{o}bius normal hyperbolic solutions} (or, in the dynamical literature: {\it homothetic solutions}) of this problem.

\subsection{The nilpotent  parabolic  M\"{o}bius group}

If in the Lie algebra $sl(2,\mathbb{C})$ we consider the nilpotent parabolic vector field,
\[
X_5 = \left(\begin{array}{cc}
     0 & 1 \\
    0 & 0  \\
    \end{array}\right), \]
the straight line $\{ t X_5 \}$ in $sl(2,\mathbb{C})$ is applied under the exponential map
into the  one-parametric subgroup of transformations associated to the nilpotent parabolic 
subgroup of M\"{o}bius matrices of ${\rm \bf  Mob}_2 \, (\widehat{\mathbb{C}})$ of the form
\[G_p (t)=
\left(\begin{array}{cc}
    1 & t \\
    0  & 1  \\
    \end{array}\right),
\]
which defines the subgroup of acting  M\"{o}bius
transformations
\begin{equation} \label{eq:Parabolic}
f_{G_p (t)} (z) = z + t,
\end{equation}
in $\mathbb{M}^2_R$.

The one parametric subgroup (\ref{eq:Parabolic}) is associated to the unitary vector field $1$ in $\widehat{\mathbb{C}}$,
which defines the first order complex differential equation in $\mathbb{M}_R^2$,
\begin{equation}\label{eq:Moebius-parabolic-field}
\dot{z} = 1.
\end{equation}

The flow is a set of horizontal parallel straight lines and there are not fixed points in  $\mathbb{M}^2_R$.

As before, we will use this one-dimensional subgroup for obtaining the set of {\it  M\"{o}bius nilpotent parabolic solutions}
 (or, in the dynamical literature: {\it parabolic solutions}).

\subsection{The Iwasawa decomposition of  ${\rm \bf SL}\, (2, \mathbb{C})$}

In the M\"{o}bius geometry of $\mathbb{M}^2_R$  (see \cite{Kisil, Kob} for details), we have the following distinguished subgroups:
\begin{itemize}
\item $K=SU(2)$ called  the {\it cyclic elliptic subgroup} of ${\rm \bf SL}\, (2, \mathbb{C})$,
\item $N = \left\{ \left(\begin{array}{cc}
    1   &  b   \\
   0   &    1   \\
    \end{array}\right) | \, b \in \mathbb{R} \right\}$  named the {\it nilpotent  parabolic subgroup} of ${\rm \bf SL}\, (2, \mathbb{C})$,
\item $B = \left\{ \left(\begin{array}{cc}
   e^{\frac{\alpha}{2}}   &  0   \\
   0   &    e^{-\frac{\alpha}{2}}   \\
    \end{array}\right) \,
     | \, \alpha \in \mathbb{R}\,  \right\}$ called the {\it normal hyperbolic subgroup} of ${\rm \bf SL}\, (2, \mathbb{C})$.
\end{itemize}

The Iwasawa decomposition Theorem states that the Lie group ${\rm \bf SL}\, (2, \mathbb{C})$ can be factorized by the before subgroups,
\begin{equation}\label{eq:Iwasawa}
{\rm \bf SL}\, (2, \mathbb{C})= B \, N \, K.
\end{equation}

This is, for any $A \in {\rm \bf SL}\, (2, \mathbb{C})$, there exist unique matrices $ \mathit{P} \in B, \mathit{R} \in N, \mathit{S} \in K$ such that
 $A$ can be factorized in the form $A = \mathit{P} \mathit{R} \mathit{S}$ (see \cite{Husemuller}, \cite{Iwa} and \cite{Kisil} for more details).

An immediate result of this  Theorem is also  the following one.

\begin{Corollary}\label{coro1:Iwasawa}
For any  M\"{o}bius transformation $f_A \in {\rm \bf Mob}_2\, (\widehat{\mathbb{C}})$ there exist unique matrices $ \mathit{P} \in B, \mathit{R} \in N, \mathit{S} \in K$ as in Iwasawa's Theorem, such that $f_A$ can be factorized in  the form $f_A = f_{\mathit{P}} \circ f_{ \mathit{R}} \circ f_{\mathit{S}}$.
\end{Corollary}

Another important result for the work is the following.

\begin{Corollary}\label{coro2:Iwasawa}
Any one-dimensional parametric subgroup  $g(t) \in {\rm \bf Mob}_2\, (\widehat{\mathbb{C}})$ can be written in a unique way of the form
\begin{equation}\label{eq:Iwasawa-decomp}
g(t) = f_{\mathit{P(t)}} \circ f_{ \mathit{R(t)}} \circ f_{\mathit{S(t)}}
\end{equation}
 with,
\[   \mathit{P(t)}=  \left(\begin{array}{cc}
    a(t)  &  b(t)   \\
  - \bar{b}(t)   &    a(t)  \\
    \end{array}\right),    \quad   \mathit{R(t)}=  \left(\begin{array}{cc}
   e^{\frac{t}{2}}   &  0   \\
   0   &    e^{-\frac{t}{2}}   \\
    \end{array}\right),    \quad  \mathit{S(t)}=  \left(\begin{array}{cc}
    1   &  t   \\
   0   &    1   \\
    \end{array}\right). \]
\end{Corollary}

\section{Equations of motion and the  M\"{o}bius solutions }\label{sec:equamotion}

In \cite{Perez} the authors obtain the equations of motion for this problem through the stereographic projection of the sphere (of radius $R$) embedded in $\mathbb{R}^3$ into the complex plane $\mathbb{C}$ endowed with the  metric (\ref{met-c}). After, the authors in \cite{Reyes} using the Vlasov-Poisson equation, obtain  the classical equation  for the motion of particles with positives masses $m_1,m_2,\cdots,m_n$  in a Riemannian or semi-Remannian manifold with coordinates $(x^1, x^2, \cdots, x^N)$ endowed with a metric $(g_{ij})$ and associated connection $\Gamma^i_{jk}$, which are  moving under the influence  of a pairwise acting potential $U$. Such equations are given by
\begin{equation}\label{eq:Vlasov-Poisson}
 \frac{D \dot{x}^i}{d t} =\ddot{x}^i+ \sum_{l,j} \Gamma^i_{lj} \dot{x}^l \dot{x}^j = \sum_{k} m_k g^{ik} \frac{\partial U}{\partial x^k},
\end{equation} 
for $i=1,2, \cdots, N$, where $\displaystyle \frac{D}{d t}$ denotes the covariant derivative and $g^{-1}=(g^{ik})$ is the inverse matrix for the metric $g$.
\begin{Remark}
We observe that in equation (\ref{eq:Vlasov-Poisson}), the left hand
side corresponds to the equations of the geodesic curves, whereas the right hand side corresponds to the gradient of the potential in the given metric. This means also that if the potential is constant, then the particles move along
geodesics.
\end{Remark}

It is well known from the  Beltrami-Bers' Theorem (see \cite{DoCarmo, Dub, Farkas}) that all connected smooth two dimensional manifold admits an atlas such that its associated  Riemannian  metric is conformal and endows it with a Riemann surface structure. The Poincar\'e-Koebe Uniformization Theorem implies that all orientable connected smooth two dimensional manifold admits a Riemannian metric with negative constant Gaussian curvature (hyperbolic structure), except the Sphere $\mathbb{S}^2$ and the torus $\mathbb{T}^2$ (see \cite{Gusevski, Farkas}). The obstruction for the existence of hyperbolic structure of  the last mentioned surfaces is, from the Gauss-Bonnet Theorem, their Euler's characteristics. In other words, almost all the surfaces are hyperbolic \cite{Gusevski}.

\begin{Definition} 
A two dimensional positive space form is a smooth connected surface with positive constant Gaussian curvature.
\end{Definition}

The Minding's Theorem states that  all the two dimensional manifolds with the same constant Gaussian curvature are locally isometric  (see \cite{DoCarmo}).
Therefore, any  two dimensional space form is locally characterized up an isometry by the sign of the curvature and  it is locally isometric to one of the standard space forms: The plane with zero Gaussian curvature, the sphere of radius $R$ with positive Gaussian curvature $\displaystyle K=\frac{1}{R^2}$ and the hyperboloid of radius $R$  with negative Gaussian curvature $\displaystyle K=-\frac{1}{R^2}$. In particular, any positive space form belongs to the isometric  class  of $\mathbb{M}^2_R$  (and therefore to  the same  differentiable class  \cite{Dub}, chapter 3).

We state in this section the equations of motion for the  $n$--body problem
in $\mathbb{M}^2_R$ chosen as the representation of the isometric class of the whole set of two dimensional  positive  space forms.

\smallskip

Let us denote by $\mathbf{z}=(z_1,z_2,\cdots,z_n) \in  \mathbb{C}^n$  the total vector position of $n$ particles with masses
$m_i>0$ located in the point $z_i$ on the space $\mathbb{M}^2_{R} \equiv \mathbb{C}$.

\smallskip

For the pair of points $z_k$ and $z_j$ in $\mathbb{M}^2_R$ we denote
the geodesic distance between them by $d(z_k,z_j)=d_{kj}$ and define
(see \cite{Perez}) the cotangent relation
\begin{equation}\label{eq:cot}
\cot_R \left(\frac{d_{kj}}{R}\right)=\frac{2(z_k \bar{z}_j+z_j
\bar{z}_k)R^2+
(|z_k|^2-R^2)(|z_j|^2-R^2)}{ 4 R^2 \, |z_j-z_k|^2 \, |R^2+ \bar{z}_j z_k|^2}. \\
\end{equation}

The singular set in  $\mathbb{M}^2_{R}$ for the  $n$--body problem
is the zero set of the equation
 \[ |z_j-z_k|^2 \, |R^2+ \bar{z}_j z_k|^2=0. \]

From here, we obtain the following singular  sets:
\begin{enumerate}
\item   The {\it singular collision set}  given by  $\Delta(C) = \cup_{kj} \,\,  \Delta(C)_{kj}$, where the set
\begin{equation}
\label{eq:collisionset} \Delta(C)_{kj}=\{\mathbf{z} = (z_1, z_2,\cdots, z_n) \in \mathbb{C}^n \, | \, z_k = z_j, \, k \neq j\}
\end{equation}
is the one obtained by the pairwise collision of the particles with masses $m_j$ and $m_k$.

\item   The {\it singular geodesic conjugated set}    given by $\Delta (A)= \cup_{kj} \, \, \Delta(A)_{kj}$, where the set
\begin{equation}
\label{eq:antipodalset} \Delta(A)_{kj}=\left\{  \mathbf{z} = (z_1,
z_2,\cdots, z_n) \in \mathbb{C}^n \, | \, z_k= \frac{-R^2}{|z_j|^2}
\, z_j, \, k \neq j \right\},
\end{equation}
is the one obtained by the pairwise antipodal points with  masses $m_j$ and $m_k$. We remark that the singularity holds in this case because such antipodal points are in reality geodesic conjugated points for an infinity of geodesic curves which indetermine the acting force of the potential.
\end{enumerate}

It is clear that the presence of geodesic conjugated points in arbitrary equations of motion (\ref{eq:Vlasov-Poisson}) on manifolds allows us to singularities of such mechanical system, by the same reason as in this particular case.

\smallskip

We define the {\it total singular set} of  the problem as
\begin{equation}
\label{eq:totalsingularset} \Delta=\Delta(C) \cup  \,  \Delta(A).
\end{equation}

Let  $\mathbf{z} =(z_1,z_2,\cdots,z_n)$ be the total vector  position of $n$ point particles with masses
$m_1,m_2,\cdots,m_n>0$ in the space $\mathbb{M}^2_{R}$, moving under the action of the potential 
\begin{eqnarray}\label{eq:potesf}
U_R &=& U_R (\mathbf{z}, \bar{\mathbf{z}}) = \frac{1}{R}\sum_{1\leq
k < j \leq n}^n m_k m_j \cot_R
\left(\frac{d_{kj}}{R}\right) \nonumber \\
&=& \frac{1}{R} \sum_{1\leq k < j \leq n}^n m_k m_j \frac{2(z_k
\bar{z}_j+z_j \bar{z}_k)R^2+ (|z_k|^2-R^2)(|z_j|^2-R^2)}{2 R \,
|z_j-z_k| \, |R^2+ \bar{z}_j z_k|}, 
\end{eqnarray}
defined in the set $(\mathbb{M}^{2}_{R})^n \setminus \Delta$.

\smallskip

A direct substitution of the equations for the geodesics curves in $ \mathbb{M}^2_{R}$ and the gradient
\begin{eqnarray}\label{eq:gradmet}
 \frac{\partial U_R}{\partial \bar{z}_k} &=&
 \sum_{j=1, j \neq k}^n \frac{m_k m_j \, ( R^2+
 |z_k|^2)(|z_j|^2+R^2)^2(R^2+\bar{z}_jz_k)(z_j-z_k)}{4 R^2 \, |z_j-z_k|^3 \, |R^2+ \bar{z}_j z_k|^3} \nonumber \\
 \end{eqnarray}
for $k=1,2, \cdots, n$, in equation (\ref{eq:Vlasov-Poisson}),  shows that the  solutions of the problem must  satisfy
the following system of second order ordinary differential equations
\begin{equation}\label{eq:motiongral}
 m_k \ddot{z}_k -\frac{2 m_k \bar{z}_k\dot{z}_k^2}{R^2+ |z_k|^2} = \frac{2}{\lambda (z_k, \bar{z}_k)} \,
 \frac{\partial U_R }{\partial \bar{z}_k}(\mathbf{z}, \bar{\mathbf{z}}),
\end{equation}
where,
\begin{equation}\label{eq:conforesf}
\lambda (z_k, \bar{z}_k)= \frac{4R^4}{(R^2+|z_k|^2)^2}
\end{equation}
is the value of conformal function for the Riemannian metric (\ref{met-c})  in the point $(z, \bar{z})$ (see \cite{Perez}).

\begin{Definition}\label{def:Moebius-solution}
 A {\it M\"{o}bius solution}  for the $n$--body problem
 in $\mathbb{M}_R^2$  is a solution $\mathbf{z} (t) =(z_1(t), z_2(t),
\cdots, z_n(t))$ of the equations of motion (\ref{eq:motiongral}) which is invariant
 under some one-dimensional subgroup $ \{ G(t) \} \subset {\rm \bf Mob}_2\, (\widehat{\mathbb{C}})$.
\end{Definition}

Due to the Iwasawa decomposition (\ref{eq:Iwasawa-decomp}) in Corollary \ref{coro2:Iwasawa}, in the following sections we will study the M\"{o}bius solutions of (\ref{eq:motiongral}) corresponding  to the subgroups $K, N$ and $B$ of  ${\rm \bf  Mob}_2 \, (\mathbb{M}_R^2)$
which  allow us to obtain a suitable classification of motions for the  $n$--body problem in a positive space form.

\section{M\"{o}bius elliptic solutions}\label{elliptic}

We start our analysis of the M\"{o}bius solutions with the so called {\it elliptic solutions}, obtained by the action of one dimensional parametric
subgroups of the third factor  $SU(2)$  in the Iwasawa decomposition (\ref{eq:Iwasawa-decomp}), some of which in fact have been
 studied in the papers \cite{Diac, Diac-Perez, Perez, Reyes-Perez}.

\begin{Definition}\label{def:relative-equilibria}
A  {\it M\"{o}bius  elliptic solution}  for $n$--body problem
(\ref{eq:motiongral}) is a solution $\mathbf{z} (t)=(z_1(t), z_2(t),
\cdots, z_n(t))$ of the equations of motion (\ref{eq:motiongral}) which is invariant
 to  any one-dimensional subgroup $ \{ A(t) \} \subset
SU(2)/ \{ \pm I\} $. 
\end{Definition}

Since the basic one-dimensional parametric subgroups
(\ref{eq:first-vector-field}), (\ref{eq:second-vector-field}) and
(\ref{eq:third-vector-field}) generate under the
composition of functions all  one-dimensional parametric subgroups $\{ A(t) \} \subset SU(2)/ \{ \pm I\}$, we have  the following result
 obtained in \cite{Reyes-Perez}.

\begin{Theorem} \label{theo:relative-equilibria}
Let be $n$ point particles with masses $m_1,m_2, \cdots, m_n>0$
moving in $\mathbb{M}_R^2$. An equivalent condition for $z(t)=(z_1(t), z_2(t), \cdots, z_n(t))$  to be a  M\"{o}bius  elliptic solution
of (\ref{eq:motiongral}) is that the
coordinates satisfy one of the following rational functional equations depending of the time.

\begin{enumerate}
\item[\bf a.] The one obtained by the action of
(\ref{eq:first-vector-field}), that is,
\begin{equation} \label{eq:rationalsystem-1}
\frac{16 R^6 \, (1+z_k^2)(R^2 z_k-\bar{z}_k )}{(R^2+ |z_k|^2)^4} =
 \sum_{j=1, j \neq k}^n \frac{ m_j \,(|z_j|^2+R^2)^2(R^2+\bar{z}_jz_k)(z_j-z_k)}{|z_j-z_k|^3 \, |R^2+ \bar{z}_j z_k|^3} 
\end{equation}
with velocity $\displaystyle \dot{z}_k= 1+z_k^2$ at each point.

\item[\bf b.]  That obtained by the action of
(\ref{eq:second-vector-field}), that is,
\begin{equation} \label{eq:rationalsystem-2}
\frac{32 \, R^6(|z_k|^2-R^2) z_k }{ (R^2+ |z_k|^2)^4  } =
 \sum_{j=1, j \neq k}^n \frac{m_j \, (|z_j|^2+R^2)^2(R^2+\bar{z}_jz_k)(z_j-z_k)}{|z_j-z_k|^3 \, |R^2+ \bar{z}_j z_k|^3}
\end{equation}
with velocity $\displaystyle \dot{z}_k= 2 i \, z_k $ at each point.

\item[\bf c.]  The one obtained by the action of
(\ref{eq:third-vector-field}), that is,
\begin{equation} \label{eq:rationalsystem-3}
\frac{16 R^6 \,  (1-z_k^2)(\bar{z}_k+R^2 z_k )}{(R^2+ |z_k|^2)^4} =
 \sum_{j=1, j \neq k}^n \frac{ m_j \,(|z_j|^2+R^2)^2(R^2+\bar{z}_jz_k)(z_j-z_k)}{|z_j-z_k|^3 \, |R^2+ \bar{z}_j z_k|^3} 
\end{equation}
with velocity $\displaystyle \dot{z}_k=i(1-z_k^2)$ at each point.
\end{enumerate}
\end{Theorem}

\begin{proof} For the first case, by straightforward computations, we have from equation
(\ref{eq:first-vector-field-1}) the equality
\begin{equation}
\ddot{z}_k= 2 \, z_k  \left(1+z_k^2 \right),
\end{equation}
which when is substituted, together with
(\ref{eq:first-vector-field-1}) into equation (\ref{eq:motiongral}),
gives us the relation (\ref{eq:rationalsystem-1}).
The proofs for the other cases are similar.
\end{proof}

The following result give us conditions on the initial positions of the particles for generating an elliptic solution of equation (\ref{eq:rationalsystem-2}). This result is equivalent to that obtained  in \cite{Florin} for the so called fixed points.

\begin{Corollary}\label{coro:cyclical-2} With the hypotesis of Theorem \ref{theo:relative-equilibria}, a  necessary and sufficient condition for  the initial positions $z_{1,0}, z_{2,0}, \cdots, z_{n,0}$ generating a   M\"{o}bius elliptic  solution for the system
(\ref{eq:motiongral}),  invariant under the Killing vector field (\ref{eq:second-vector-field}), is that the coordinates satisfy the following system of algebraic equations,
\begin{equation} \label{eq:rationalsystem-4}
\frac{32 \, R^6(|z_{k,0}|^2-R^2) z_{k,0} }{ (R^2+ |z_{k,0}|^2)^4  } =
 \sum_{j=1, j \neq k}^n \frac{m_j \, (|z_{j,0}|^2+R^2)^2(R^2+\bar{z}_{j,0}z_{k,0})(z_{j,0}-z_{k,0})}{|z_{j,0}-z_{k,0}|^3 
\, |R^2+ \bar{z}_{j,0} z_{k,0}|^3}
\end{equation}
and the velocity in each particle is given by the relation $ \dot{z}_{k,0} = 2 i z_{k,0}$, for $k=1,2,\cdots, n$.
\end{Corollary}

\begin{proof} Let $w_k = w_k(t)= e^{2 it} \, z_{k,0}$ be the action of the Killing vector field (\ref{eq:second-vector-field}) in the initial condition point $z_{k,0}$, with velocity $ \dot{z}_{k,0} = 2 i z_{k,0}$. If we multiply equation (\ref{eq:rationalsystem-4})  by the number $ e^{2it}$ in both sides, and use the equality $\bar{w}_j(t) w_k (t) = \bar{z}_{j,0} e^{-it} \, z_{k,0}e^{it} = \bar{z}_{j,0} \, z_{k,0}$, then we obtain the system,
\begin{equation} 
\frac{32 \, R^6(|w_k|^2-R^2) w_k }{ (R^2+ |w_k|^2)^4} =
 \sum_{j=1, j \neq k}^n \frac{m_j \, (|w_j|^2+R^2)^2(R^2+\bar{w}_j w_k)(w_j-w_k)}{|w_j-w_k|^3 \, |R^2+ \bar{w}_j w_k|^3},
\end{equation}
which shows that $w_k(t)$ is a solution of (\ref{eq:rationalsystem-2}).

The converse claim follows directly if in system (\ref{eq:rationalsystem-2}) we put $t=0$. This proves the Corollary.
\end{proof}

 In \cite{Diac, Diac-Perez} the authors show examples of M\"{o}bius  elliptic  solutions invariant under the Killing vector field (\ref{eq:second-vector-field-1})  on the sphere embedded  in $\mathbb{R}^3$, but when all the particles are sited  for all time on the same circle obtained of intersecting with an horizontal plane. In \cite{Perez} you can find examples in the two and three body problems defined on $\mathbb{M}^2_{R}$. 

The following result, which is our version of the  principal axis Theorem of Euler in $\mathbb{M}_R^2$,  shows that such that for studyng the whole set of elliptic M\"{o}bius solutions it is sufficient with studying those invariant under
 the Killing vector field (\ref{eq:second-vector-field-1}).

\begin{Lemma}\label{lema:principal}  The solutions of  equations of motion (\ref{eq:motiongral}) invariants under the Killing vector fields  
(\ref{eq:first-vector-field-1}) and  (\ref{eq:third-vector-field-1}) can be carried isometrically  into those solutions of  (\ref{eq:motiongral}) invariant under the  Killing vector field  (\ref{eq:second-vector-field-1}).
\end{Lemma}

\begin{proof} 

For this is sufficient with showing that the Killing vector  (\ref{eq:first-vector-field-1}) and  (\ref{eq:third-vector-field-1}) are conjugated with the one  (\ref{eq:second-vector-field-1}).

We give the proof for the Killing vector fields (\ref{eq:first-vector-field-1})  and  (\ref{eq:second-vector-field-1}). 
The other case follows in the same way.

Consider an element 
\begin{equation}
 A= \left(\begin{array}{cc}
    a        &  b     \\
   -\bar{b} & \bar{a}   \\
    \end{array}\right) \in {\rm SU} (2)
\end{equation}
and the corresponding isometric  M\"{o}bius transformation in $\mathbb{M}^2_{R}$,
\begin{equation}
w=\frac{az + b}{-\bar{b}z+ \bar{a} }
\end{equation}
which carries the foliation of solutions of the Killing vector field (\ref{eq:first-vector-field}) into those  solutions of the isometric vector field (\ref{eq:second-vector-field}).

This is, such that transformation carries the Killing vector field $\dot{z}= 2 i z$ into the one $\dot{w}= 1+w^2$, if and only if,
\begin{equation}
 \frac{2 i z}{( -\bar{b}z+ \bar{a})^2 }= \frac{\dot{z}}{( -\bar{b}z+ \bar{a})^2 }=\dot{w}= 1+w^2 =  \frac{(az+ b)^2 + ( -\bar{b}z+\bar{a})^2} {( -\bar{b}z+ \bar{a})^2 }, 
\end{equation}
which, for $ -\bar{b}z+\bar{a} \neq 0$,  allows us to the equation
\begin{equation}\label{eq:quadratic-euler} 
2 i z= (a^2+ \bar{b}^2) z^2 + 2( a b-\bar{a}\bar{b}) z+ (b^2+ \bar{a}^2).  
\end{equation}

Since equation (\ref{eq:quadratic-euler}) must hold for all $z$ in and open subset of $\mathbb{C}$ then we obtain the pair of equations
\begin{eqnarray}\label{eq:pair-euler}
 a^2+ \bar{b}^2 &= & 0 \nonumber \\
 a b-\bar{a}\bar{b} &=&   i.  \nonumber \\
\end{eqnarray}

The factors $a-i  \bar{b}=0$ and $a+ i \bar{b}=0$ in the first equation of the system (\ref{eq:pair-euler})  can be verified are equivalent to $w(0) = i$ and $w(\infty)=- i$. If we multiply by $a$ in both sides of  the second equation of compatibility in the same system  (\ref{eq:pair-euler}) and use the first equation of that system and   the equality $|a|^2 + |b|^2 = 1$, it becomes into the first factor $ i a+ \bar{b}=0$.

 If we put $a=\alpha_1 + i \alpha_2$ and $b= \beta_1+ i \beta_2$, then such that equations implies that $b=\alpha_2 + i \alpha_1$.

Therefore, the searched  M\"{o}bius transformation has associated any rotation matrix
\begin{equation}\label{eq:matrix-euler}
 A= \left(\begin{array}{cc}
   \alpha_1+ \alpha_2 i                                   &    \alpha_2+  i \alpha_1    \\
  - \alpha_2+  i \alpha_1 &   \alpha_1-  \alpha_2 i     \\
    \end{array}\right) \in {\rm SU} (2),
\end{equation}
which can be seen as an element in $SO(2) \subset {\rm SU} (2)$ with $ \displaystyle \alpha_1^2  + \alpha_2^2 = \frac{1}{2}$.

Conclusion is given by the Existence and Uniqueness Theorem for the complex differential equation associated to the given vector field
 (\ref{eq:second-vector-field-1}) with  initial conditions in the fixed  points  $w=-i$ and $w=i$.

This ends the proof.
\end{proof}

\begin{Remark}\label{unique} The above method in Lemma \ref{lema:principal} will be  employed for proving the existence of  M\"{o}bius solutions invariant under the Killing vector fields  (\ref{eq:first-vector-field-1}) and (\ref{eq:third-vector-field-1}),  by carrying isometrically them into the Killing vector field (\ref{eq:second-vector-field-1}).
\end{Remark}

\subsection{Examples of   M\"{o}bius  elliptic solutions in the $n$-body problem}

We show here several types of solutions for the $n$-body problem by using the algebraic system (\ref{eq:rationalsystem-4}).
This is, for the initial real positions $z_k(0)= z_{k,0}$ and with velocities
$\dot{z}_{k,0}= 2 i \, z_{k,0}$ we construct, by using the Corollary \ref{coro:cyclical-2}, periodic circular orbits for obtaining solutions which generalize those obtained in \cite{Perez} for equal masses.

\smallskip

From here now on in this section, the words {\it inside of the geodesic circle $|z|= R$} means that in the
sphere of radius $R$ embedded in $\mathbb{R}^3$ the corresponding motion is realized in the south-hemisphere  while the words
{\it outside of the geodesic circle $|z|= R$} means that the corresponding motion is realized in the north-hemisphere (see \cite{Reyes-Perez}). On other hand, the word {\it degenerated solution} means that for small perturbation on the mass ratio the system changes on the number of M\"{o}bius elliptic 
solutions close to such that solution. Moreover, this property of {\it non-degeneracy} of the solutions must be related to the {\it stability} of the solutions  in the sense of \cite{Florin2}, since they are obtained under conditions of transversality (see \cite{Guille}).

\smallskip

We will use the following notation along this section for short in the results.
For $z \in \mathbb{M}_R^2$ we denote by
\begin{enumerate}
\item[\bf a.] $P_S$ to the {\it south pole}  $z= 0$.
\item[\bf b.] $T_S$ to the {\it southern tropic}  $|z|= (\sqrt{2}-1)R$.
\item[\bf c.] $E_R$ to the {\it geodesic circle}  $|z|= R$.
\item[\bf d.] $T_N$ to the {\it northern tropic}  $\displaystyle |z|= \frac{R}{\sqrt{2}-1}$
\item[\bf e.] $P_N$ to the {\it north pole}  $z= \infty$.
\item[\bf f.] $ \Omega_1$ to the open region from $P_S$ to the southern tropic $T_S$,
\[ \Omega_1 = \{ z \in \mathbb{M}_R^2 \quad | \quad 0 < |z| < (\sqrt{2}-1)R \}.  \]
\item[\bf g.] $ \Omega_2$ to the open region from $T_S$ to the geodesic circle $E_R$, 
\[ \Omega_2 = \{ z \in \mathbb{M}_R^2 \quad | \quad (\sqrt{2}-1)R  < |z| < R \}.  \]
\item[\bf h.] $ \Omega_3$ to the open region from $E_R$ to the northern tropic $T_N$, 
\[ \Omega_3 = \{ z \in \mathbb{M}_R^2 \quad | \quad R < |z| < \displaystyle \frac{R}{\sqrt{2}-1} \}.  \]
\item[\bf i.] $ \Omega_4$ to the open region from $T_N$ to the north pole $P_N$, 
\[ \Omega_4 = \{ z \in \mathbb{M}_R^2 \quad | \quad \displaystyle \frac{R}{\sqrt{2}-1} < |z| \}.  \]
\end{enumerate}

\subsubsection{\bf New examples of M\"{o}bius elliptic solutions in the two-body problem  in $\mathbb{M}_R^2$}

For the two-body case, we will  show  solutions of ({\ref{eq:motiongral}) invariant under the vector field (\ref{eq:second-vector-field-1}) such that the two masses move along two different circles.

Firstly we observe that the system (\ref{eq:rationalsystem-4}) for the two body problem becomes  into the simple algebraic system (not depending 
on the time $t$),
\begin{eqnarray} \label{eq:rationalsystem-two-body}
\frac{32 \, R^6(|z_1|^2-R^2) z_1 }{ (R^2+ |z_1|^2)^4  } &=&
 \frac{m_2 \, (|z_2|^2+R^2)^2(R^2+\bar{z}_2z_1)(z_2-z_1)}{|z_2-z_1|^3 \, |R^2+  z_1 \bar{z}_2|^3}, \nonumber \\
 \frac{32 \, R^6(|z_2|^2-R^2) z_2 }{ (R^2+ |z_2|^2)^4  } &=&
 \frac{m_1 \, (|z_1|^2+R^2)^2(R^2+\bar{z}_1z_2)(z_1-z_2)}{|z_1-z_2|^3 \, |R^2+ \bar{z}_1 z_2|^3}, \nonumber \\
\end{eqnarray}

By Corollary \ref{coro:cyclical-2},  a necessary and sufficient condition for the
existence of a  M\"{o}bius solution invariant under the Killing vector field (\ref{eq:second-vector-field})
is that the initial conditions $z_1(0)=\alpha$, $z_2(0)=\beta$ satisfy equations (\ref{eq:rationalsystem-two-body}).

Moreover, it can be seen that in equations (\ref{eq:rationalsystem-two-body}), $z_1=\alpha$ is a real number, if and only if, $z_2=\beta$ is also a real number (see \cite{Perez, Reyes-Perez}). We obtain the following result.

\begin{Lemma}\label{lema:new-equilibria} Let $0 <\alpha < R$ be a real number. For the two body problem
in $\mathbb{M}^2_R$ with equal masses, a necessary  condition for the
existence of one M\"{o}bius elliptic solution invariant under the Killing vector field (\ref{eq:second-vector-field}),
is that, for the initial real positions $z_1(0)=\alpha$ and  $z_2(0)=\beta$ with corresponding velocities
$\dot{z}_1(0)= 2 i \, \alpha$ and $\dot{z}_2(0)= 2 i \, \beta$, holds one of the some conditions
\begin{enumerate}
\item[\bf 1.] $\beta_1 =-\alpha$ inside of the geodesic circle, or its geodesic conjugated point $\displaystyle \beta_2= \frac{R^2}{\alpha}$ outside of the geodesic circle.
\item[\bf 2.] For $\displaystyle \beta_3= \frac{R(\alpha-R)}{\alpha+R}$ inside of the geodesic circle, or its geodesic conjugated point $\displaystyle \beta_4=- \frac{R(R+\alpha)}{\alpha-R}$ outside of the geodesic circle,.
\end{enumerate}
\end{Lemma}

\begin{proof} In order of finding solutions for equation (\ref{eq:rationalsystem-two-body}), we search those where both particles are
sited in a same geodesic meridian along all the isometric action. 
 
 This is, without of generality we can put $z_1=\alpha$ in the real axis, and
$z_2= \beta=\lambda \alpha$ for some suitable real  value of $\lambda$, which will be founded with algebraic methods. 

For this case we obtain that $R^2+ \bar{z}_1 z_2=R^2+ \bar{z}_2 z_1$, and then, avoiding collisions and conjugated points, when we divide the left hand sides and the right hand ones of equation (\ref{eq:rationalsystem-two-body}), it becomes into the equation
\begin{equation}\label{eq:single-two-body-1}
0 =\alpha [ m_1 (R^2-\alpha^2)(R^2+\lambda^2\alpha^2)^2 + m_2 \lambda(R^2-\lambda^2\alpha^2)(R^2+\alpha^2)^2],
\end{equation}
and for  $\alpha\neq0$, we obtain the simple equation in the unknown $\lambda$,
\begin{equation}\label{eq:single-two-body-2}
0 =m_1 (R^2-\alpha^2)(R^2+\lambda^2\alpha^2)^2 + m_2 \lambda(R^2-\lambda^2\alpha^2)(R^2+\alpha^2)^2.
\end{equation}

For the algebraic equation (\ref{eq:single-two-body-2}) with equal masses, the solutions for the unknown $\lambda$ are given by
\[ \lambda_1 =-1, \quad \lambda_2= \frac{R^2}{\alpha^2}, \quad \lambda_3= \frac{R(\alpha-R)}{\alpha(\alpha+R)},
\quad \lambda_4=- \frac{R(R+\alpha)}{\alpha(\alpha-R)},  \]
which given the value of $\alpha$, gives us the initial conditions for $\beta$:
\[ \beta_1 =-\alpha, \quad \beta_2= \frac{R^2}{\alpha}, \quad \beta_3= \frac{R(\alpha-R)}{\alpha+R},
\quad \beta_4=- \frac{R(R+\alpha)}{\alpha-R},  \]

A simple analysis on the positions of such values for $\beta$ proves the the corresponding positions in the items of the Lemma, which ends the proof.
\end{proof}

\smallskip

We obtain the main result of this subsection, the necessary and sufficient conditions for the
existence of classes of  M\"{o}bius elliptic solutions invariant under the Killing vector field (\ref{eq:second-vector-field}) 
for the two body problem.

\begin{Theorem}\label{theo:new-equilibria-main} For the two body problem in $\mathbb{M}^2_R$ with fixed equal masses $m$ we have the following class of 
 M\"{o}bius elliptic solutions,
\begin{enumerate}
\item[\bf 1.] For the mass ratio $\displaystyle  m < 2 R^3$,
\begin{enumerate}
\item[\bf a.] There exists a unique initial condition $0< \alpha_1 <(\sqrt{2}-1)R$ for one circular motion inside $ \Omega_1$,
and for this condition there are two distinct non degenerate circular motions with real initial conditions, 
\begin{enumerate}
\item[\bf i.] $\beta_{1,1} = -\alpha_1$ on the same circle inside $ \Omega_1$.
\item[\bf ii.] $\displaystyle  \beta_{1,2} = \frac{R^2}{\alpha_1}$ for one circle sited in $ \Omega_4$.
\end{enumerate}
\item[\bf b.] There exists a unique initial condition $(\sqrt{2}-1)R < \alpha_2 < R$ for a circular motion inside $ \Omega_1$  
and for this condition, there are two  distinct non degenerate circular motions with initial conditions,
\begin{enumerate}
\item[\bf i.] $\beta_{2,1} = -\alpha_2$ on the same circle inside $ \Omega_2$.
\item[\bf ii.] $\displaystyle \beta_{2,2} = \frac{R^2}{\alpha_2}$ sited in a different circle inside of $ \Omega_3$.  
\end{enumerate}
\end{enumerate} 
\item[\bf 2.] For the mass ratio $\displaystyle  m  <2 R^3$,
\begin{enumerate}
\item[\bf a.] There exists a unique initial condition $0< \alpha_3 <(\sqrt{2}-1)R$ for one circular motion inside of $ \Omega_1$,
and for this condition there exist two distinct non degenerate circular motions with real initial conditions, 
\begin{enumerate}
\item[\bf i.] $\displaystyle \beta_{3,1}=  \frac{R(\alpha_1-R)}{\alpha_1+R}$ sited on different  circle, but inside of $ \Omega_1$.
\item[\bf ii.] $\displaystyle \beta_{3,2}=  \frac{R(\alpha_1+R)}{R-\alpha_1}$ sited on a different  circle inside of  $ \Omega_4$. 
\end{enumerate}
\item[\bf b.] There exists a unique initial condition  $(\sqrt{2}-1)R < \alpha_4<R$ for one circular motion inside of $ \Omega_2$,
and for this condition there are two distinct non degenerate circular motions with real initial conditions, 
\begin{enumerate}
 \item[\bf i.] $\displaystyle \beta_{4,1}= \frac{R(\alpha_2-R)}{\alpha_2+R}$ sited on different circle, but inside of $ \Omega_2$.
 \item[\bf ii.]$\displaystyle \beta_{4,2}=  \frac{R(\alpha_2+R)}{R-\alpha_2}$ sited on different circle and  inside of $ \Omega_3$.  
\end{enumerate}
\end{enumerate}
\item[\bf 3.] For the the mass ratio  $\displaystyle  m = 2 R^3$, there exist the unique initial condition $\alpha_{tan}=(\sqrt{2}-1)R$ for a circular motion along $T_S$, and for this condition, there exist two distinct degenerate solutions with real initial conditions,
\begin{enumerate}
\item[\bf a.] $\beta_{{\rm tan},1}=-(\sqrt{2}-1)R$ sited in the same circle $T_S$.  
\item[\bf b.] $\displaystyle \beta_{{\rm tan},2}=\frac{R}{(\sqrt{2}-1)}$ sited in the circle $T_N$.  
\end{enumerate}
\end{enumerate}

When the  mass ratio   $\displaystyle  m > 2 R^3$ there are not M\"{o}bius elliptic solutions for this problem.
\end{Theorem}

\begin{proof} We will use the necessary condition given by Lemma \ref{lema:new-equilibria}. In order of prove the sufficient condition, we separate 
the proof in all the cases.
\begin{enumerate}
\item[\bf 1. a.] For the cases $\beta= -\alpha$ and $\displaystyle \beta = \frac{R^2}{\alpha}$, when we substitute in any equation of system (\ref{eq:rationalsystem-two-body}) we obtain the single and same equation,
\begin{equation}\label{eq:first-single-two-body}
\frac{|R^2-\alpha^2| \alpha }{ (R^2+ \alpha^2)^2  } = \frac{\sqrt[3]{m}}{4\sqrt[3]{2} \, R^2}
\end{equation}
which, since the solutions $\beta= -\alpha$ and $\displaystyle \beta = \frac{R^2}{\alpha}$ of this equation appear together with their geodesic conjugates, 
as can be directly checked, we will analyse equation (\ref{eq:first-single-two-body}) for $ 0 < \alpha < R$.

We define the smooth functions,
\begin{equation}\label{eq:first-function-single-two-body}
F(\alpha)= \frac{ (R^2-\alpha^2) \alpha }{ (R^2+ \alpha^2)^2},  \quad G(\alpha) = \frac{\sqrt[3]{m}}{4\sqrt[3]{2} \, R^2}.
\end{equation}
on the interval $[0,R]$.

The function $F$  vanishes in the extremes of the interval and is increasing in the interval $(0, (\sqrt{2}-1)R)$, decreasing in 
$((\sqrt{2}-1)R,R)$, with an absolute maximum 
\begin{equation}
F\left((\sqrt{2}-1)R\right)= \frac{1}{4R}.  
\end{equation}

This implies that, if 
\begin{equation}
F\left((\sqrt{2}-1)R\right)- G\left((\sqrt{2}-1)R\right) = \frac{1}{4R}  - \frac{\sqrt[3]{m}}{4\sqrt[3]{2} \, R^2} >0
\end{equation}
or  equivalently,
$\displaystyle  m < 2 R^3$, the graphs of the functions $F$ and $G$ intersect transversally and in exactly two points, say $(\alpha_1, F(\alpha_1))=(\alpha_1, G(\alpha_1))$ and $(\alpha_2, F(\alpha_2))=(\alpha_2, G(\alpha_2))$ with the corresponding arguments $0 < \alpha_1 < (\sqrt{2}-1)R < \alpha_2 <R $, which define the initial positions for the aforementioned solutions {\bf i} and {\bf ii}. This proves the claim of items
{\bf 1. a.} and {\bf 1. b}. 

On the other hand, if $\displaystyle  m = 2 R^3$, the functions $F$ and $G$ intersect tangentially  and in exactly the point 
$\displaystyle \left((\sqrt{2}-1)R,\frac{1}{4R} \right)$ which define the initial position for the respective degenerate solutions 
{\bf a.} and {\bf b.} of {\bf 3.} (see \cite{Perez}). 

\item[\bf 2. a.] If $\displaystyle \beta = -\frac{R(R-\alpha)}{R+\alpha}$ or $\displaystyle \beta = \frac{R(R+\alpha)}{R-\alpha}$, when we substitute in any equation of system (\ref{eq:rationalsystem-two-body}) we obtain the equation
\begin{equation}\label{eq:second-single-two-body}
\frac{(R^2-\alpha^2) \, \alpha}{(R^2+ \alpha^2)^2} =\frac{ m }{8 R^4}.  
\end{equation}

Since again the solutions $\displaystyle \beta = -\frac{R(R-\alpha)}{R+\alpha}$ and $\displaystyle \beta = \frac{R(R+\alpha)}{R-\alpha}$ of this equation appear together with their geodesic conjugates, as can be again directly checked, we analyse equation (\ref{eq:second-single-two-body}) for the case $ 0 < \alpha < R$.

If we define the smooth functions $F(\alpha)$ and $\displaystyle G(\alpha)= \frac{ m }{8 R^4 }$ as in the item {\bf 1.}, with the same analysis we obtain that
\begin{equation}
F\left((\sqrt{2}-1)R\right)- G\left((\sqrt{2}-1)R\right) = \frac{1}{4R}  - \frac{ m }{8 R^4 } >0
\end{equation}
or  equivalently,
$\displaystyle  m < 2 R^3$, the graphs of the functions $F$ and $G$ intersect transversally and in exactly two points, say again $(\alpha_3, F(\alpha_3))=(\alpha_3, G(\alpha_3))$ and $(\alpha_4, F(\alpha_4))=(\alpha_4, G(\alpha_4))$ with the found arguments $0 < \alpha_3 < (\sqrt{2}-1)R < \alpha_4 <R $, which define the initial positions for the corresponding solutions {\bf i} and {\bf ii}. This proves the claims {\bf 2. a.} and  {\bf 2. b.} 
 \end{enumerate}
 
 On the other hand, if $\displaystyle  m = 2 R^3$, the functions $F$ and $G$ intersect tangentially  and in exactly the point 
$\displaystyle \left((\sqrt{2}-1)R,\frac{1}{4R} \right)$ which define the initial position for the corresponding  degenerate solutions
{\bf a.} and {\bf b.} of {\bf 3.} (see again \cite{Perez}).

 It is clear that for the mass ratio  $\displaystyle  m > 2 R^3$ there are not M\"{o}bius elliptic solutions of the problem. This ends the proof.
 \end{proof}

\begin{Corollary} There are only teen different class of  M\"{o}bius elliptic solutions for the $2$-body problem in $\mathbb{M}_R^2$. Eight of them are non-degenerate  while four of them are.
\end{Corollary}

The Theorem \ref{theo:new-equilibria-main} proves the existence of new orbits on different hemispheres and shows a rich dynamics in this problem. In other words, by using the  M\"{o}bius geometry we  show new   M\"{o}bius  elliptic solutions for the $n$--body problem in $\mathbb{M}_R^2$. We have the following important result.

\begin{Corollary} There are  M\"{o}bius elliptic solutions for the $n$-body problem in $\mathbb{M}_R^2$ invariant under the Killing vector fields
(\ref{eq:first-vector-field-1}) and (\ref{eq:third-vector-field-1}). 
\end{Corollary}

\subsubsection{\bf M\"{o}bius elliptic solutions for the Eulerian three-body problem in $\mathbb{M}_R^2$}

In \cite{Perez} the authors found the so called elliptic Eulerian solutions for the initial positions
$z_{1}(0)=\alpha, z_{2}(0)=0, z_{3}(0)=\beta$
with $\alpha$,$\beta$ real numbers. They suppose that $z_{1}$ and  $z_{3}$ have mass $m$, and $z_{2}$ has mass $M$.

The algebraic system of equations (\ref{eq:rationalsystem-4}) becomes for this case,
\begin{eqnarray}\label{elliptic-eulerian-trhee-body}
\dfrac{32R^{6}(\alpha ^{2}-R^{2})\alpha}{(R^{2}+\alpha^{2})^{4}} &=& \dfrac{-M\alpha}{\vert \alpha \vert ^{3}} + \dfrac{m(\beta ^{2}+R^{2})^{2}(R^{2}+\beta \alpha)(\beta-\alpha)}{\vert \beta - \alpha \vert ^{3}\vert R^{2}+\beta \alpha \vert ^{3}}, \nonumber \\
0 &=& \dfrac{\alpha (\alpha^{2}+R^{2})^{2}}{\vert \alpha \vert^{3}}+\dfrac{\beta(\beta ^{2}+R^{2})^{2}}{\vert \beta \vert^{3}}, \nonumber \\
\dfrac{32R^{6}(\beta ^{2}-R^{2})\beta}{(R^{2}+\beta^{2})^{4}} &=& \dfrac{-M\beta}{\vert \beta \vert ^{3}} + \dfrac{m(\alpha ^{2}+R^{2})^{2}(R^{2}+\alpha \beta)(\alpha-\beta)}{\vert \alpha - \beta \vert ^{3}\vert R^{2}+\alpha \beta \vert ^{3}}. \nonumber \\
\end{eqnarray}

From the second equation in the system (\ref{elliptic-eulerian-trhee-body}) is necessary that $\alpha\beta < 0$, and without lose of generality we can suppose that $\beta <0 <\alpha$. Therefore such equation become into the one,
\begin{equation}
\beta \alpha ^{2}+(R^{2}+\beta ^{2})\alpha +\beta R^{2}=0,
\end{equation}
which has the solutions,
\begin{eqnarray}
\alpha = -\beta \quad \mbox{\rm for}  \quad \vert \beta \vert < R , \nonumber \\
\alpha = \dfrac{-R^{2}}{\beta} \quad \mbox{\rm for} \quad \vert \beta \vert > R.  \nonumber \\
\end{eqnarray}

Since the second equality is a geodesic conjugated algebraic solution for $\alpha$, then unique solution is the antisymmetric $\alpha = -\beta$ inside the geodesic
circle $|z|=R$. 

We have the following result on this M\"{o}bius configuration of the Eulerian three body problem.

\begin{Corollary}{\bf (to Theorem 5.1 of \cite{Perez})}\label{cor:eulerian-three-body} For the fixed mass with mass ratio $\displaystyle 4R^3  > \frac{M}{2}+ \frac{m}{4}$ there are only four distinct possible class of  non degenerate solutions for the Eulerian M\"{o}bius three-body problem in $\mathbb{M}_R^2$ with one particle  of mass $M$ sited in the origin of coordinates and the other two of mass $m$ sited in opposite sides of the same circle. There is one of them in each region $ \Omega_i$ of $\mathbb{M}_R^2$. For the case $\displaystyle 4R^3  = \frac{M}{2}+ \frac{m}{4}$ are only two distinct degenerated solutions, one along $T_S$ and the other along $T_N$. In the case $\displaystyle  4R^3  < \frac{M}{2}+ \frac{m}{4}$ such configuration does not exist.
\end{Corollary}

\begin{proof}If we substitute $\alpha = -\beta$ both in the first as and the third equations of system (\ref{elliptic-eulerian-trhee-body}), we obtain the same following equation
\begin{equation}\label{eq:elliptic-eulerian-single}
32\left(\dfrac{\alpha R^{2}(R^{2}-\alpha^{2})}{(R^{2}+\alpha^{2})^{2}}\right)^{3}= M\left(\dfrac{R^{2}-\alpha ^{2}}{R^{2}+\alpha ^{2}}\right)^{2} + \dfrac{m}{4}.
\end{equation}

If $\alpha$ is a solution of the algebraic equation (\ref{eq:elliptic-eulerian-single}), a simple substitution proves that the geodesic conjugated real number $\displaystyle  - \frac{R^2}{\alpha} < -R$ also is a solution of such that equation, and such number is outside of the geodesic circle $z=R$.

Therefore it will be sufficient with considering the problem for $0 < \alpha < R$, and for this we consider the smooth real functions
\begin{eqnarray}\label{eq:elliptic-eulerian-functions}
F(\alpha) &=& 32\left(\dfrac{\alpha R^{2}(R^{2}-\alpha^{2})}{(R^{2}+\alpha^{2})^{2}}\right)^{3}, \nonumber \\
G(\alpha) &=& M\left(\dfrac{R^{2}-\alpha ^{2}}{R^{2}+\alpha ^{2}}\right)^{2} + \dfrac{m}{4}, \nonumber \\
\end{eqnarray}
in the interval $[0, R]$.

For the first function $F$ in the pair (\ref{eq:elliptic-eulerian-functions}), we have that it is increasing in the interval
$\displaystyle [0, (\sqrt{2} -1)R)$ and decreasing in  the interval $((\sqrt{2} -1)R, R]$, having a maximum value 
$F((\sqrt{2} -1)R)= 4R^3$. 

On the other hand, the second function $G$ is strictly decreasing in the interval $[0, R]$ and has the value   $\displaystyle G((\sqrt{2} -1)R)= \frac{M}{2}+ \frac{m}{4}$.

Therefore, if we define the function $H(\alpha)= F(\alpha)-G(\alpha)$, then 
\begin{eqnarray}\label{eq:elliptic-eulerian-functions-valuated}
H(0) &=& F(0)-G(0) = - \left(\frac{M}{2}+ \frac{m}{4} \right) < 0,  \nonumber \\
H((\sqrt{2} -1)R) & =& F((\sqrt{2} -1)R)-G((\sqrt{2} -1)R) \nonumber \\
 &=& 4R^3- \left(\frac{M}{2}+ \frac{m}{4} \right) >0, \nonumber \\
\end{eqnarray}
by hypothesis. 
 
 It follows  that there exists $\alpha_1 \in (0, (\sqrt{2} -1)R)$ such  $H(\alpha_1)=0$, and this is a solution for (\ref{eq:elliptic-eulerian-single}). 

Moreover, $F(R)=0$ and $\displaystyle G(R)=\frac{m}{4}$, and then $\displaystyle H(R)=-\frac{m}{4} <0$, which implies that there exists 
$\alpha_2 \in ((\sqrt{2} -1)R, R)$ such that $H(\alpha_2)=0$, this is, $\alpha_2$ is also solution of (\ref{eq:elliptic-eulerian-single}).

The transversal intersection of the smooth functions $F$ and $G$ on the interval $[0,R]$ for this case shows the uniqueness of such solutions, which determine the positions for the required solutions. 

If the intersection is not transversal then we obtain the single point $\alpha= (\sqrt{2} -1)R$ where the functions $F$ and $G$ are tangential, which determine the position of the particle for obtaining the unique degenerated solution inside the geodesic circle.

The last claim is obvious and this ends the proof.
\end{proof}

\subsubsection{\bf M\"{o}bius  elliptic solutions in the four-body problem  in $\mathbb{M}_R^2$}

We give an example for the the Four-body problem in $\mathbb{M}_R^2$, by considering four particles all of equal masses $m$, which are sited in the early on the ellipse
$\displaystyle \frac{x^{2}}{\alpha^{2}}+\frac{y^{2}}{\beta^{2}}=1$, with the initial conditions
$z_{1}(0)=\alpha, z_{2}(0)=-\alpha, z_{3}(0)=i\beta, z_{4}(0)=-i\beta$, and with the condition $ 0<\beta \leq \alpha $.

The algebraic system of equations (\ref{eq:rationalsystem-4}), after reducing, are for this case,
\begin{eqnarray}\label{eq:elliptic-four-body}
\dfrac{32R^{6}(R^{2}-\alpha ^{2})\alpha}{(R^{2}+\alpha ^{2})^{4}} &=& \dfrac{m(R^{2}+\alpha ^{2})^{2}(R^{2}-\alpha ^{2})}{4\alpha ^{2}\vert R^{2}-\alpha ^{2} \vert ^{3}} + \dfrac{2\alpha m (R^{2}+\beta ^{2})^{2}(R^{2}-\beta ^{2})}{\vert \alpha +i\beta \vert ^{3} \vert R^{2}+ i\alpha \beta \vert ^{3}}, \nonumber \\
\dfrac{32R^{6}(R^{2}-\beta ^{2})\beta}{(R^{2}+\beta ^{2})^{4}} &=& \dfrac{m(R^{2}+\beta ^{2})^{2}(R^{2}-\beta ^{2})}{4\beta ^{2}\vert R^{2}-\beta ^{2} \vert ^{3}} + \dfrac{2\beta m (R^{2}+\alpha ^{2})^{2}(R^{2}-\alpha ^{2})}{\vert \alpha +i\beta \vert ^{3} \vert R^{2}+ i\alpha \beta \vert ^{3}}. \nonumber \\
\end{eqnarray}

It is clear that $\alpha=\beta$ satisfies the pair of equations (\ref{eq:elliptic-four-body}), which implies the following result 
(see \cite{Florin} and \cite{Diacu1}).

\begin{Corollary}{\bf (to Theorem 1 of \cite{Diacu1})} For the four-body problem with equal masses on the same circle and with the configuration of square, up a geodesic conjugation, there are only six class of solutions. Four are  non degenerate with two inside the geodesic circle and two outside. The other two are degenerate, one is sited inside the geodesic circle and the other is outside.  
\end{Corollary}

\begin{proof} 

The proof follows the same methodology as in Theorem \ref{theo:new-equilibria-main} and Corollary \ref{cor:eulerian-three-body}.

If we put $\beta=\alpha$ in any of the equations  (\ref{eq:elliptic-four-body}), we obtain, after reducing it, 
\begin{equation}\label{eq:elliptic-four-body-single}
\dfrac{32R^{6} \alpha^3}{m(R^{2}+\alpha ^{2})^6} = \frac{1}{ 4\vert R^{2}-\alpha ^{2} \vert ^{3}} + 
\dfrac{1}{\sqrt{2 \, (\alpha^{4} + R^{4})^3}}
\end{equation}

If $\alpha$ is a solution of the algebraic equation (\ref{eq:elliptic-four-body-single}), another  simple substitution proves that the geodesic conjugated real number $\displaystyle - \frac{R^2}{\alpha} < -R$  is also a solution of such that equation, and such number is outside of the geodesic circle 
$z=R$.

Therefore it will be sufficient with considering the problem for $0 < \alpha < R$, and for this we consider the smooth real functions
\begin{eqnarray}\label{eq:elliptic-four-body-functions}
F(\alpha) &=& \dfrac{32R^{6} \alpha^3}{m(R^{2}+\alpha ^{2})^6}, \nonumber \\
G(\alpha) &=& \frac{1}{ 4 (R^{2}-\alpha ^{2})^{3}} + 
\dfrac{1}{\sqrt{2 \, (\alpha^{4} + R^{4})^3}}, \nonumber \\
\end{eqnarray}
in the interval $[0, R)$.

For the first function $F$ in the pair (\ref{eq:elliptic-four-body-functions}), we have that it is increasing in the interval
$\displaystyle \left[0, \frac{\sqrt{3} R}{3}\right)$ and decreasing in  the interval $\displaystyle \left(\frac{\sqrt{3} R}{3}, R\right]$, having a maximum value 
$\displaystyle F\left(\frac{\sqrt{3} R}{3}\right)= \frac{81 \sqrt{3}}{128 m R^3}$. 

On the other hand, the second function  $G$ is also increasing in such interval with a singular point in $\alpha = R$  and  has the value   $\displaystyle G\left(\frac{3 R}{\sqrt{3}}\right)= \frac{27}{4 R^6} \left( \frac{1}{8} + \frac{1}{5 \sqrt{5}}  \right)$.

Therefore, if we define again the function $H(\alpha)= F(\alpha)-G(\alpha)$, then 
\begin{eqnarray}\label{eq:elliptic-eulerian-functions-valuated}
H(0) &=& F(0)-G(0) = - \left(\frac{1}{\sqrt{2}} + \frac{1}{4} \right)\frac{1}{R^6} < 0,  \nonumber \\
H\left(\frac{\sqrt{3} R}{3}\right) & =& F\left(\frac{\sqrt{3} R}{3}\right)-G\left(\frac{\sqrt{3} R}{3}\right) \nonumber \\
&=& \frac{81 \sqrt{3}}{128 m R^3}- \frac{27}{4 R^6} \left( \frac{1}{8} + \frac{1}{5 \sqrt{5}} \right)>0, \nonumber \\
\end{eqnarray}
and for the necessary mass  condition $\displaystyle R^3 > \frac{128 m}{81 \sqrt{3}} \left( \frac{1}{\sqrt{2}} + \frac{1}{4}\right)$ there exists $\displaystyle \alpha_1 \in \left(0, \frac{\sqrt{3} R}{3}\right)$ such  $H(\alpha_1)=0$, and this is, as before,  a solution for (\ref{eq:elliptic-four-body-single}). 

On the other hand, since  $F(\alpha)$ is strictly increasing in the interval $\displaystyle \left(\frac{\sqrt{3} R}{3}, R\right)$, $G(\alpha)$ is strictly decreasing in such interval and 
\begin{eqnarray}
\lim_{\alpha \to R^{-}} F(\alpha) &=& \frac{1}{2m R^3}, \nonumber \\
\lim_{\alpha \to R^{-}} G(\alpha) &=& + \infty, \nonumber \\
\end{eqnarray}
then there exists $\displaystyle \theta \in \left(\frac{\sqrt{3} R}{3}, R\right)$ such that $H(\theta)<0$. Therefore there exists 
$\displaystyle \alpha_2 \in \left(\frac{\sqrt{3} R}{3}, \theta \right)$ such $H(\alpha_2)=0$. This is, $\alpha_2$ is a second solution of (\ref{eq:elliptic-four-body-single}).

Once again, as before, the transversality of functions $F$ and $G$ on the interval $[0,R)$ shows the uniqueness of such solutions in this interval, which proves that for the necessary mass condition $\displaystyle R^3 > \frac{128 m}{81 \sqrt{3}} \left( \frac{1}{\sqrt{2}} + \frac{1}{4}\right)$  there are two solutions for the equation (\ref{eq:elliptic-four-body-single}) in the interval $(0,R)$.

On the other hand, for the mass condition $\displaystyle \left(\frac{1}{\sqrt{2}} + \frac{1}{4} \right)m > \frac{81 \sqrt{3}  R^3}{128 }$, there are not intersection between the functions $F$ and $G$.

Therefore, from continuity of the problem respect to parameters, there is  at least one mass condition such that the intersection between $F$ and $G$ is in a single point.

In order of finding  this sufficient condition for the having elliptic M\"{o}bius  solutions for this problem, it is necessary to find the tangential argument $\displaystyle \alpha \in \left(0, \frac{\sqrt{3}R}{3} \right)$ such that 
\begin{eqnarray}
F(\alpha)-G(\alpha) &=&  0, \nonumber \\
F'(\alpha)-G'(\alpha) &=&  0, \nonumber \\
\end{eqnarray}
in terms of the parameters $R,m$, say $\displaystyle \alpha_{\rm tan} = \alpha(R,m)$. The existence of this tangential point $\alpha_{\rm tan} $ follows  from the concavities of the functions $F(\alpha)$ and $G(\alpha)$. The uniqueness follows from the fact that the function $F(\alpha)$ has only one inflection point in the interval $\left(0, \frac{\sqrt{3}R}{3} \right)$, and the existence of more tangential points must imply a bigger number of inflection points, as a straightforward  calculus shows.

Therefore, we obtain the following cases.
\begin{enumerate}
\item For the mass ratio $F(\alpha_{\rm tan})-G(\alpha_{\rm tan}) >0$ there exist two solutions $\alpha_1$, $\alpha_2$ of equation (\ref{eq:elliptic-four-body-single}) in the interval $(0, R)$ such that $0 < \alpha_1 < \alpha_{\rm tan} <\alpha_2$. Such solutions are the  initial conditions for two non degenerate solutions of this four-body problem. 
\item For the mass ratio $F(\alpha_{\rm tan})-G(\alpha_{\rm tan})  = 0$ there exist a unique solution $\alpha = \alpha_{\rm tan}$ of equation (\ref{eq:elliptic-four-body-single}) in the interval $(0, R)$. Such solution is  the  initial conditions for one degenerate solution of this four-body problem. 
\item For the mass ratio $F(\alpha_{\rm tan})-G(\alpha_{\rm tan})  <0$ there exist are not solution for this four-body problem. 
\end{enumerate}

This ends the proof.
\end{proof}

\section{M\"{o}bius hyperbolic solutions}\label{hyperbolic}

 We show  the  M\"{o}bius solutions generated by the
one dimensional hyperbolic subgroup of ${\rm \bf  Mob}_2 \, (\mathbb{M}_R^2)$
in the Iwasawa decomposition (\ref{eq:Iwasawa}).

\smallskip

 In one Riemannian manifold, a
pair of points $p_1$ and $p_2$ are geodesic conjugated if there exists at least a pair
of different geodesics joining  them. It is well known that in
$\mathbb{M}^2_R$ any pair of antipodal points $\displaystyle z_k=
\frac{-R^2}{|z_j|^2} \, z_j $ are geodesic conjugated and the whole space is
foliated by all geodesic curves passing through such pair of
points. The set of all such curves is called the {\it geodesic
conjugated class foliation}.

\begin{Definition} A solution $\mathbf{z}(t)=(z_1(t), z_2(t), \cdots, z_n(t))$ of equation (\ref{eq:motiongral})
is called  homothetic if all the particles move on curves whose path
belong to the same geodesic conjugated class  foliation.
\end{Definition}

Proceding  as in Lemma \ref{lema:principal}, up an isometry,   we can assume that one point is
the origin of coordinates $z=0$ with conjugated point $z=\infty$,
and the geodesic foliation of $\mathbb{M}^n_R$ is the
 set of straight lines passing through such point called {\it meridians}.

We remark that since the geodesics are always parametrized such that their tangent vectors
have constant speed (see \cite{DoCarmo}), then one particle moving along one homothetic solution not necessarily
does it in a geodesic way.

Among the whole set of homothetic solutions of
(\ref{eq:motiongral}) is the subclass generated by the
action of the  one-parametric subgroup of hyperbolic  M\"{o}bius
transformations  (\ref{eq:Hyperbolic}).

\begin{Definition} An homothetic solution $\mathbf{z}(t)=(z_1(t), z_2(t), \cdots, z_n(t))$ of equation (\ref{eq:motiongral})
is called {\it M\"{o}bius hyperbolic} if it is invariant under the vector field (\ref{eq:hyperbolic-vector-field}).
\end{Definition}

We state the following result for this type of dolutions in the $n$-body problem, which follows from direct substitutions.

\begin{Theorem} Let be $n$ point particles
with masses $m_1,m_2, \cdots, m_n>0$ moving in $\mathbb{M}^2_{R}$. An equivalent  condition for $\mathbf{z}(t)=(z_1(t), z_2(t), \cdots, z_n(t))$ 
 to be a   M\"{o}bius hyperbolic solution of (\ref{eq:motiongral}), is that the
coordinates satisfy the system of rational functional functions (depending on the time $t$),
\begin{equation} \label{eq:rationalsystem-hyp}
\frac{8 \,R^6(R^2-|z_k|^2) z_k }{(R^2+ |z_k|^2)^4} =
 \sum_{j=1, j \neq k}^n \frac{m_j \, (|z_j|^2+R^2)^2(R^2+\bar{z}_jz_k)(z_j-z_k)}{|z_j-z_k|^3 \, |R^2+ \bar{z}_j z_k|^3}
\end{equation}
and the velocity in each particle is given by the relation $ \dot{z}_k (t)= z_k (t)$,
for $k=1,2,\cdots, n$. All the solutions $z_k=z_k(t)$ are
backward asymptotic to the origin of coordinates, which implies that
there are not collisions between the particles.
\label{thm:existence-hyp}
\end{Theorem}

{\bf Proof.} If we derive the vector field (\ref{eq:hyperbolic-vector-field}), we obtain $\displaystyle
\ddot{z}_k = z_k$, which when is substituted into equations
of motion (\ref{eq:motiongral}) allows to the system
(\ref{eq:rationalsystem-hyp}).  This ends the proof of the Theorem.
\qed

\smallskip

For the two body problem, if we suppose that two particles with masses $m_1$ and $m_2$ and positions $z_1(t)$ and $z_2(t)$  in $\mathbb{M}^2_{R}$ are moving as a M\"{o}bius hyperbolic solution, in \cite{Reyes-Perez} is shown that, the masses of the particles are equal,  if and only if, $z_1(t)=-z_2(t)$.  
It follows that, up an isometry, for the two-body problem with equal masses there is only a  single class of M\"{o}bius hyperbolic solutions.

In \cite{Diac-Perez}, the authors show that a necessary and
sufficient condition for having a {\it Lagrangian} homothetic solution in the $3$--body problem in the unitary sphere embedded in $\mathbb{R}^3$, 
 is that the configuration be always an equilateral triangle and that the masses (always sited in the same horizontal plane intersecting the sphere)  be equal. Therefore, in order to study in beyond this kind of motion in $\mathbb{M}_R^2$, for the  M\"{o}bius hyperbolic solutions case,  it is enough to analyse the case of equal masses.

\section{M\"{o}bius  nilpotent parabolic solutions}\label{parabolic}

Finally,  we show the  M\"{o}bius nilpotent parabolic  solutions corresponding to the second factor $N$  in the
Iwasawa decomposition (\ref{eq:Iwasawa}),  and associated to the
one dimensional subgroup (\ref{eq:Parabolic})  of ${\rm \bf  Mob}_2 \, (\mathbb{M}_R^2)$.

\begin{Definition} A solution $\mathbf{z}(t)=(z_1(t), z_2(t), \cdots, z_n(t))$ of equation (\ref{eq:motiongral})
is called {\it  M\"{o}bius nilpotent parabolic} if it is invariant under the parabolic vector field (\ref{eq:Moebius-parabolic-field}).
\end{Definition}

We can now state the following result whose proof follows again  by
straightforward substitutions.

\begin{Theorem}\label{thm:existenceklein2}  Let be $n$ point particles
with masses $m_1,m_2, \cdots, m_n>0$ moving in $\mathbb{M}^2_{R}$. An equivalent  condition for $\mathbf{z}(t)=(z_1(t), z_2(t), \cdots, z_n(t))$ 
 to be a   to be a M\"{o}bius nilpotent parabolic solution of system
(\ref{eq:motiongral})  is that the coordinate functions, depending of the time $t$, satisfy the functional (also depending on the time $t$)
equations
\begin{equation} \label{eq:condrationalsystem-parabolic}
-\frac{16 R^6 \bar{z}_{k}}{(R^2+ |z_k|^2)^4} = \sum_{\substack{j=1\\
j\ne k}}^n \frac{m_j \, (R^2+ |z_j|^2)^2(R^2+ z_k
\bar{z}_j)(z_j-z_k)}{|z_j-z_k|^3 \, |R^2+ \bar{z}_j z_k|^3},
\end{equation}
and with  corresponding velocities $\dot{z}_k (t) = 1$, for $k=1,\dots, n$.
\end{Theorem}

 In \cite{Diacu2} the authors have shown that there are not M\"{o}bius nilpotent parabolic solutions for the $n$--body problem with negative curvature.
 In \cite{Reyes} it is shown the non existence of true M\"{o}bius parabolic solutions for the $n$-body problem in $\mathbb{H}^2_R$. We remark that in $\mathbb{M}^2_{R}$ these two types of conic motions coincide (although they not be isometries as in $\mathbb{H}^2_R$), and unfortunately the same happens for their existence in this positive case.

\begin{Corollary}\label{thm:no-existence-parabolic}
There are no one class of  M\"{o}bius nilpotent  parabolic  solutions for the
$n$--body problem in $\mathbb{M}^2_{R}$.
\end{Corollary}

{\bf Proof.} Let be $n$ point particles of masses $m_1,\dots,
m_n$ moving on $\mathbb{M}^2_{R}$ with total vector position $\mathbf{z} (t)=(z_1 (t),\dots, z_n (t))$, 
with $z_k=z_k(t)$ satisfying  equations (\ref{eq:condrationalsystem-parabolic}).

Since due to the action of the vector field (\ref{eq:Moebius-parabolic-field}) the real part of the solutions become positive for suitable
large values of time, then  applying a translation to the whole set
of solutions if it is necessary, we can assume that the $k$-th
particle reaches the imaginary axis. This is,  we  assume $\mbox{Re} \,
(z_k) =0$, and $\mbox{Re} \, (z_j)  \geq 0$ for all $z_j \neq z_k$ with  $j=1,\dots, n$.

Therefore, the real parts in each side of system
(\ref{eq:condrationalsystem-parabolic}) are
\begin{eqnarray}\label{eq:real-part-zero}
0 &=& -\mbox{Re} \, \left(\frac{16 R^6 \bar{z}_{k}}{(R^2+ |z_k|^2)^4} \right) \nonumber \\
  &=&  \sum_{\substack{j=1\\ j\ne k}}^n \frac{m_j \, (R^2+ |z_j|^2)^2}{|z_j-z_k|^3 \, |R^2+ \bar{z}_j z_k|^3} \,
\mbox{Re} \,[(R^2+ z_k \bar{z}_j)(z_j-z_k)] \nonumber \\
&=&  \sum_{\substack{j=1\\ j\ne k}}^n \frac{m_j \, (R^2+
|z_j|^2)^2}{|z_j-z_k|^3 \, |R^2+ \bar{z}_j z_k|^3} \, (R^2+ |z_k|^2)
\mbox{Re} \,[z_j]. \nonumber \\
\end{eqnarray}

It follows, from the chain of equalities (\ref{eq:real-part-zero}) that the whole set of particles must be also located
on  the imaginary axis, that is, $\mbox{Re} \, (z_j(0)) = 0$ for all $ j=1,2, \cdots, n$.

Therefore, if we put $\displaystyle z_l (t)=
t+ i \beta_l$ for the position of the $l$-th particle,
then we obtain the equalities
\begin{equation}\label{eq:parabolic-parametrized}
 (R^2+ z_k \bar{z}_j)(z_j-z_k)= t \left(\beta_j - \beta_k \right)^2
+  \left( t^2+ \beta_j \beta_k +R^2 \right) \left(\beta_j
- \beta_k \right) \, i.
\end{equation}

When we substitute equations (\ref{eq:parabolic-parametrized})  in the system
(\ref{eq:condrationalsystem-parabolic}), we obtain for the real parts the relations
\begin{equation} \label{eq:real-part-parabolic}
- \frac{16 R^6}{(R^2+ |z_k|^2)^4} = \sum_{\substack{j=1\\
j\ne k}}^n \frac{m_j \, (R^2+ |z_j|^2)^2}{|z_j-z_k|^3 \, |R^2+
\bar{z}_j z_k|^3} \, \left(\beta_j - \beta_k \right)^2,
\end{equation}
for $k, j =1,\cdots, n$, and for all $t \in \mathbb{R}$.

System (\ref{eq:real-part-parabolic}) never holds, since if we avoid
collisions and geodesic conjugated points, the left hand side is always negative, whereas the
right hand side is positive. This contradiction proves the
Theorem. \qed.

\section{Conclusions}\label{sec:conlusions}

 As in \cite{Reyes}, we have that unique M\"{o}bius solutions for the $n$-body problem in $\mathbb{M}^2_{R}$ are either the elliptic, the hyperbolic  or the composition of them (called loxodromic). Again, as we have pointed there, for the complete study of this type of solutions, it is sufficient with considering the Cartan-Haussdorf decomposition $KBK$ of $SL(2,\mathbb{C})$. Also as in that reference, in words of the classical mechanics we have the
following result.

\begin{Corollary}\label{coro:equilibria}
 The unique relative equilibria for the two dimensional positively curved $n$-body problem are those  M\"{o}bius elliptic solutions.
\end{Corollary}

\subsection*{Acknowledgements}

The second author Reyes-Victoria acknowledges once again to the project PLAN DE FORTALECIMIENTO DEL GRUPO DE INVESTIGACION EN ECUACIONES DIFERENCIALES 066-2013 of the UNICARTAGENA, Colombia, for the total support in the realization of this second paper with the professor Ortega-Palencia.


\begin{thebibliography}{10}

\bibitem{Florin} Diacu, F., Relative equilibria of the curved $n$-body problem, Atlantis studies in Dynamical Systems, vol. 1, Atlantis Press, Paris, 2012. MR3185362.

\bibitem{Florin2} Diacu, F., Martinez, R., P\'erez-Chavela, E., Sim\'o, C., On the stability of tetrahedral relative equilibria in the 
positively curved 4-for body problem, Physica D, {\bf 256-257}, pp 21-35, 2013.

\bibitem{Diac} Diacu, F., P\'erez-Chavela, E., Santoprete, M.,
The n-body problem in spaces of constant curvature. Part I: Relative
Equilibria {\it Journal of Nonlinear Science} {\bf 22}, 247-266,
(2012).

\bibitem{Diac-Perez} Diacu, F., P\'erez-Chavela, E., Homographic solutions in the curved three-body problem, {\it Journal
of Differential Equations}, {\bf 250}, 340-366, (2011).

\bibitem{Diacu2} Diacu, F., P\'erez-Chavela, E., Reyes, J.G., An intrinsic approach in
the curved $n$--body problem. The negative case. {\it Journal of Differential Equations}, {\bf 252}, 4529-4562, (2012).

\bibitem{Diacu1} Diacu, F., Thorn, B.,  Rectangular orbits of the curved four-body problem, Proc. Amer. Math. Soc., {\bf 143}, pp 1583-1593, 2015,

\bibitem{DoCarmo} Do Carmo , M., Differential Geometry of Curves and
Surfaces, Prentice Hall, New Jersey, USA, 1976.

\bibitem{Dub} Dubrovine, B., Fomenko, A., Novikov, P. Modern Geometry,
Methods and Applications, Vol. I, II and III, Springer-Verlag, 1984, 1990.

\bibitem{Farkas} Farkas, H. M., Kra, I., Riemann Surfaces,
Springer-Verlag, ISBN 10-038-79770-31, 1988.

\bibitem{Guille} Guillemin, V., Golubitsky, M., Stable mappings and their  singularities, Springer-Verlag, New York, 1974.

\bibitem{Gusevski} Gusevskii, N.A., "Uniformization", in Hazewinkel, Michiel, Encyclopedia of Mathematics, Springer, 2001.

\bibitem{Husemuller} Husemoller, D.,  Joachim,  M.,  Jurco, B.,  Schottenloher,  M., Gram-Schmidt Process, Iwasawa Decomposition,
 and Reduction of Structure in Principal Bundles, Basic Bundle Theory and K-Cohomology Invariants,
Lecture Notes in Physics Volume 726, pp 189-201, 2008.

\bibitem{Iwa} Iwasawa, K., On some types of topological groups, Ann. of Math. , 50,  pp. 507-558, 1949.

\bibitem{Kisil}  Kisil, Erlangen Program at Large-1: Geometry of invariants, {\bf SIGMA} 6, 2010.

\bibitem{Kob} Kobayashi, S., Transformation Groups in Differential Geometry
Springer. ISBN 3-540-05848-6, 1970

\bibitem{Kozlov} V.V.~Kozlov and A.O.~Harin, Kepler's problem in constant
curvature spaces, {\it Celestial Mech.~Dynam.~Astronom} {\bf 54}
(1992), 393-399.

\bibitem{Perez} P\'erez-Chavela, E., Reyes-Victoria, J.G.,
An intrinsic approach in the curved $n$--body problem. The positive
curvature case. {\it Transactions of the American Math Society},
{\bf 364}, 3805-3827, (2012).

\bibitem{Reyes-Perez}  P\'erez-Chavela, E., Reyes-Victoria, J.G.,
M\"{o}bius solutions of  the curved $n$--body problem for positive
curvature, arXiv:1207.0737, 2012.

\bibitem{Reyes} Reyes-Victoria, J.G., Ortega-Palencia, P.P., The Vlasov-Poisson equation, the M\"{o}bius geometry and the $n$-body problem in a negative space form, Preprint 2014, arXiv:1408.111, 2015. 

\bibitem{Hans} Schwerdtfeger, H., Geometry of the complex numbers, ed. Dover Publications Inc., USA, 1979.

\bibitem{Sharpe} Sharpe, R., W., Differential Geometry, {\it Cartan's generalization of Klein's Erlangen program},  GTM, Springer-Verlag, 1991.


\end{thebibliography}
\end{document}